\documentclass{daj}
\usepackage{amsfonts,amsmath,latexsym,color,epsfig,enumitem,amssymb,amsthm}
\newtheorem{theorem}{Theorem}
\newtheorem{lemma}{Lemma}

\newtheorem{claim}{Claim}

\newcommand{\eps}{\varepsilon}
\newcommand{\bZ}{\mathbb{Z}}
\newcommand{\ol}[1]{\overline{#1}}
\newcommand{\pr}{\mathrm{Pr}}
\newcommand{\cS}{\mathcal{S}}
\newcommand{\cR}{\mathcal{R}}
\newcommand{\cP}{\mathcal{P}}
\newcommand{\cU}{\mathcal{U}}

\dajAUTHORdetails{%
  title = {A Question of Erd\H{o}s and Graham on Egyptian Fractions}, 
  author = {David Conlon, Jacob Fox, Xiaoyu He, Dhruv Mubayi, Huy Tuan Pham, Andrew Suk, and Jacques Verstra\"ete},
  plaintextauthor = {David Conlon, Jacob Fox, Xiaoyu He, Dhruv Mubayi, Huy Tuan Pham, Andrew Suk, and Jacques Verstraete},
  plaintexttitle = {A Question of Erdos and Graham on Egyptian Fractions}, 
  runningtitle = {A Question of Erd\H{o}s and Graham on Egyptian Fractions}, 
  runningauthor = {Conlon, Fox, He, Mubayi, Pham, Suk, Verstra\"ete},
  copyrightauthor = {David Conlon, Jacob Fox, Xiaoyu He, Dhruv Mubayi, Huy Tuan Pham, Andrew Suk, and Jacques Verstra\"ete},
  keywords = {Egyptian fractions, entropy methods, subset sums.},
}   

\dajEDITORdetails{%
   year={2025},
   number={28},
   received={25 April 2024},   
   published={19 December 2025},  
   doi={10.19086/da.154329},       
}   

\begin{document}

\begin{frontmatter}[classification=text]

\title{A Question of Erd\H{o}s and Graham on Egyptian Fractions} 

\author[dc]{David Conlon\thanks{Supported by NSF Awards DMS-2054452 and DMS-2348859.}}
\author[jf]{Jacob Fox\thanks{Supported by NSF Awards DMS-2154129 and DMS-2452737.}}
\author[xh]{Xiaoyu He\thanks{Supported by NSF Award DMS-2103154.}}
\author[dm]{Dhruv Mubayi\thanks{Supported by NSF Awards DMS-1952767 and DMS-2153576.}}
\author[htp]{Huy Tuan Pham\thanks{Supported by a Clay Research Fellowship and a Stanford Science Fellowship.}}
\author[as]{Andrew Suk\thanks{Supported by an NSF CAREER Award and NSF Awards DMS-1952786 and DMS-2246847.}}
\author[jv]{Jacques Verstra\"ete\thanks{Supported by NSF Awards DMS-1800332 and DMS-2347832.}}

\begin{abstract}
Answering a question of Erd\H{o}s and Graham,  we show that for each fixed positive rational number $x$ the number of ways to write $x$ as a sum of reciprocals of distinct positive integers each at most $n$ is $2^{(c_x + o(1))n}$ for an explicit constant $c_x$ increasing with $x$.
\end{abstract}
\end{frontmatter}

\section{Introduction}

The study of Egyptian fractions, that is, sums of reciprocals of distinct positive integers, has a long history in combinatorial number theory (see, for example,~\cite{BE22}). The fact that every positive fraction can be written as an Egyptian fraction goes back at least to work of Fibonacci at the start of the 13th century. Much more recently, a result of Bloom~\cite{Bl21} says that any subset of the natural numbers of positive upper density has a finite subset the sum of whose reciprocals adds to 1.

In this paper, we will be concerned with a problem raised by Erd\H{o}s and Graham~\cite[Page 36]{ErGr} in 1980 (see also~\cite[Problem 297]{Bl}): how many ways are there to write 1 as a sum of distinct unit fractions with denominator at most $n$? In 2017, this problem was partially addressed by several users, including Lucia, RaphaelB4 and js21, in response to a MathOverflow post~\cite{T17} by Mikhail Tikhomirov. They showed that there is an explicit constant $c \approx 0.91117$ such that the number of such Egyptian fractions is at most $2^{(1+o_n(1))c n}$. This already answered one particular question of Erd\H{o}s and Graham, who asked whether the answer was $2^{n-o(n)}$. Here we solve Erd\H{o}s and Graham's problem more precisely by showing that the count is equal to $2^{(1+o_n(1))cn}$ for the same $c \approx 0.91117$.



Our main result more generally estimates the number of Egyptian fractions summing to any fixed positive rational. Let $h:[0,1]\to \mathbb{R}$ be given by $h(p) = -p\log_2 p - (1-p)\log_2(1-p)$ for $p\in (0,1)$ and $h(0)=h(1)=0$.

\begin{theorem}\label{thm:main-count}
    For any fixed $x \in \mathbb{Q}_{>0}$, 
    the number of subsets $A\subseteq [n]$ with $x = \sum_{a\in A}1/a$ is $2^{c_x n + o(n)}$, where
    \[
        c_x := \int_{0}^{1}h\left(\frac{1}{1+e^{\lambda/y}}\right)dy
    \] 
    and $\lambda$ is the unique real number such that 
    \[
        \int_{0}^{1}\frac{1}{y(1+e^{\lambda/y})}dy = x.
    \]
    In particular, $c_x$ is a strictly increasing function with $c_0=0$, $c_1 \approx 0.91117$ and $c_x \to 1$ as $x \to \infty$.
\end{theorem}

Our proof has two main steps. In the first step, we use entropy methods to show that $2^{(c_x + o(1))n}$ is the correct asymptotic count for the number of Egyptian fractions formed by adding distinct unit fractions with denominator at most $n$ whose sum is at most $x$. This may also be obtained using the techniques discussed in the MathOverflow post~\cite{T17}, but we include our proof, which we found before being alerted to that post, for completeness. Then, in the second step, we use a method reminiscent of the absorption technique in extremal graph theory to show that the same asymptotic count holds for the number of Egyptian fractions summing to exactly $x$. Very roughly, we first set aside a small reservoir subset of $[n]$. Then, after finding many subsets of $[n]$ disjoint from this reservoir whose sums of reciprocals are somewhat smaller than $x$ and whose elements have no very large prime power factors, we  iteratively `clean' these fractions by adding unit fractions from the reservoir to obtain a sum $x'<x$ with small denominator. This is accomplished through the use of a recent result~\cite{CFP} on the existence of homogeneous generalized arithmetic progressions in subset sums. Finally, we find a small subset of the remaining reservoir whose sum of reciprocals is equal to the remaining difference $x-x'$. 

\section{Counting through entropy}

In this section, we will use entropy methods to estimate the number of subsets $A\subseteq [n]$ with $s(A) := \sum_{a\in A}1/a \le x$. To state our result, we need some notation. As in the introduction, let $h(p)=-p\log_2 p - (1-p)\log_2 (1-p)$ for $0<p<1$ and $h(0)=h(1)=0$. Given $x > 0$, we choose $p_1,\dots,p_n \in [0,1]$ so as to maximize $\sum_{m=1}^{n} h(p_m)$ given $\sum_{m=1}^{n} p_m/m \le x$. We then let $P(x)$ be the distribution of $(Y_1,\dots,Y_n)$, where each $Y_m=1$ with probability $p_m$ and $0$ otherwise independently of each other. Then the (Shannon) entropy of $P(x)$ is given by $H(P(x)) = \sum_{m=1}^{n}h(p_m)$.

\begin{lemma}\label{lem:main-ent}
    For $\eps>0$ and $x \in (0,(1-\eps)(\ln n)/2)$, the number of subsets $A\subseteq [n]$ with $s(A) \le x$ is bounded above by $2^{H(P(x))}$. 
    Furthermore, there is $c_{x,n}>0$ with $c_{x,n} = \Theta(e^{-2x})$ such that, for any $U \subseteq [n]$, the number of subsets $A\subseteq U$ with $s(A) \le x$ is bounded below by $2^{H(P(x)) - (n-|U|) - O(\sqrt{n/c_{x,n}})}$.
\end{lemma}

We first characterize $P(x)$.

\begin{lemma}\label{lem:opt-dis}
    For $x < (1/2)(\sum_{m=1}^{n} 1/m)$, the distribution $P(x)$ is given by setting $p_m = \frac{1}{1+e^{c_{x,n} n/m}}$ for the unique $c_{x,n}$ such that $\sum_{m=1}^{n} p_m/m = x$. 
    Moreover, there is $C>0$ such that, provided $x < (1-\eps)(\ln n)/2$, $c_{x,n} > 0$ and $c_{x,n} \in [C^{-1}e^{-2x}, Ce^{-2x}]$. 
    Finally, for $\eps>0$ and assuming $x\in [\eps,1/\eps]$, $c_{x,n} = \lambda+o_n(1)$, where $\lambda > 0$ is the unique solution to 
    \[
        \int_{0}^{1} \frac{1}{y(1+e^{\lambda/y})} dy = x.
    \]
    In particular, 
    \begin{equation}\label{eq:ent-cal}
        H(P(x)) = (1+o_n(1)) n \int_{0}^{1} h\left(\frac{1}{1+e^{\lambda/y}}\right)dy = \Theta(n).
    \end{equation}
    Note that here the $o_n(1)$ terms may depend on $\eps$. 
\end{lemma}

\begin{proof}
    Note that $H(P(x))$ is strongly convex and bounded (as a function of $p_1,\dots,p_n$) in $[0,1]^n$. The unique stationary point of $H(P(x))$ in $[0,1]^n$ is $p_1=\dots=p_n=1/2$. Thus, for $x<(1/2)(\sum_{m=1}^{n}1/m)$, the maxima of $H(P(x))$ must be achieved on the boundary of $[0,1]^n \cap \{\sum_m p_m/m \le x\}$ and, since $h(x)=0$ for $x\in \{0,1\}$, we must have that the maxima are achieved on $\sum_m p_m/m = x$. At a maximum $(p_1^*,\dots,p_n^*)$, we must have that $\nabla H(P(x))$ is parallel to $(1,1/2,\dots,1/n)$. Noting that $h'(p) = \log \frac{1-p}{p}$, we obtain that such a point satisfies $p_m^* = \frac{1}{1+e^{c'/m}}$ for some $c'$. By the condition $\sum_{m\le n}p_m^*/m=x$, we must also have that $c'$ satisfies 
    \[
        \sum_{m=1}^{n} \frac{1}{m(1+e^{c'/m})}=x.
    \]

    Letting $c = c_{x,n} = c'/n$, we have that 
    \[
        \frac{1}{n} \sum_{m=1}^{n} \frac{1}{(m/n)(1+e^{c/(m/n)})} = x.
    \]
    It is easy to check that $c > 0$ when $x < (1-\eps)(\ln n)/2$ and, for $x$ sufficiently large, that $c = \Theta(e^{-2x})$. Indeed, for the last estimate, we observe that, for $k \ge 1$,   
    \[
        \sum_{m=cn/(k+1)}^{cn/k} \frac{1}{m(1+e^{cn/m})} = \Theta(e^{-k}k^{-1}),
    \]
    so \[
    \sum_{m\le cn} \frac{1}{m(1+e^{cn/m})} = O(1).
    \]
    For 
    $1 \le k \le 1/c$, 
    \[
    \sum_{m=cnk}^{cn(k+1)} \frac{1}{m(1+e^{cn/m})} \le \frac{1}{k(1+e^{1/(k+1)})}, \qquad \sum_{m=cnk}^{cn(k+1)} \frac{1}{m(1+e^{cn/m})} \ge \frac{1}{(k+1)(1+e^{1/k})}.
    \]
    Note that $1+e^{1/k} = 2 + O(1/k)$, so we obtain that \[
    \sum_{k= 1}^{1/c} \frac{1}{k(1+e^{1/(k+1)})}, \; \sum_{k=1}^{1/c}\frac{1}{(k+1)(1+e^{1/k})} = \frac{1}{2}\ln (1/c) + O(1),
    \]
    from which we immediately deduce the desired estimate on $c$. 
    
    Finally, for $x\in [\eps,1/\eps]$ and $c>0$, we can approximate $\frac{1}{n} \sum_{m=1}^{n} \frac{1}{(m/n)(1+e^{c/(m/n)})}$ by the integral $\int_{0}^{1} \frac{1}{y(1+e^{c/y})}dy$. We thus obtain that for $\lambda>0$ satisfying $\int_{0}^{1} \frac{1}{y(1+e^{\lambda/y})} dy = x$, we have $c_{x,n} = \lambda+o_n(1)$, from which (\ref{eq:ent-cal}) follows readily. 
\end{proof}

We will use the following version of the standard Berry--Esseen bound \cite{B, E, S} in the proof of Lemma \ref{lem:main-ent}. 

\begin{lemma}\label{lem:B-E}
    Let $X_1,\dots,X_n$ be independent centered random variables with $\mathbb{E}[X_i^2] = \zeta_i$ and $\mathbb{E}[|X_i|^3] = \rho_i$. Let $Z = \frac{\sum_{i=1}^{n}X_i}{(\sum_{i=1}^{n}\zeta_i)^{1/2}}$ and $Z'$ be a standard Gaussian. Then 
    \[
        \sup_{y\in \mathbb{R}} |\pr[Z \le y] - \pr[Z' \le y]| \le C\frac{\sum_{i=1}^{n}\rho_i}{(\sum_{i=1}^{n}\zeta_i)^{3/2}}.
    \]
\end{lemma}

\begin{proof}[Proof of Lemma \ref{lem:main-ent}]
Let $S$ be a finite set of real numbers and $n=|S|$. Let $r_S(x)$ be the number of subsets of $S$ that sum to at most $x$. Define the random variable $X$ to be a uniform random subset of $S$ whose elements sum to at most $x$. Note that the entropy of $X$ satisfies $H(X)=\log_2 r_S(x)$, so $r_S(x)=2^{H(X)}$. For each $s \in S$, let $X_s$ be the indicator random variable of the event $s \in X$ and let $p_s = \Pr[X_s]$, so that $\mathbb{E}[\sum_{s\in S}s X_s]=\sum_{s \in S} sp_s \leq x$. Observe that $X$ has the same distribution as the joint distribution of the $n$ random variables $X_s$. Therefore, by subadditivity of the entropy function, we have 
\[
    H(X) \leq \sum_{s \in S}H(X_s)=\sum_{s \in S}h(p_s).
\]
Hence, we get the upper bound 
$$r_S(x) \leq 2^h,$$ 
where $h$ is the maximum value of $\sum_{s \in S} h(p_s)$ over all choices of $(p_s)_{s \in S}$ satisfying $\sum_{s \in S} sp_s \leq x$. In particular, for $S=\{1/m:m\in [n]\}$, we obtain that the number of subsets $A\subseteq [n]$ with $s(A)\le x$ is at most $2^{H(P(x))}$, as claimed.

We now turn to the lower bound. Consider independent Bernoulli random variables $Y_m$ for $m \in [n]$ satisfying $Y_m=1$ with probability $p_m$ and $Y_m=0$ otherwise. Let $Y=(Y_m)_{m\in [n]}$, $Z=\sum_{m\in [n]} Y_m/m$ and $E$ be the indicator of the event $Z \le x$. Recall that, for $a\in \{0,1\}$, the conditional entropy $H(Y | E=a) = -\sum_{y\in \{0,1\}^n} \pr[Y=y | E=a] \log \pr[Y=y | E=a]$ and $H(Y | E) = \sum_{a\in \{0,1\}} \pr[E=a] H(Y | E=a)$. Since $E$ is determined by $Y$, we have 
\[
    H(Y) = H(Y,E) = H(Y|E) + H(E) = H(E) + \sum_{a\in \{0,1\}} \pr[E=a] H(Y | E=a).
\]
We thus have
\begin{equation}\label{eq:ent-decomp}
H(Y | E=1) = \frac{1}{\pr[Z\le x]}\left(H(Y) - H(E) - \Pr[Z>x] H(Y | E=0)\right). 
\end{equation}
Let $c=c_{x,n}$ as in Lemma \ref{lem:opt-dis}. By that lemma, the random variable $Z$ has variance 
\begin{align}\label{eq:2nd-mm}
    \sum_{m=1}^{n} \left(\frac{1}{(1+e^{cn/m})m^2} - \frac{1}{(1+e^{cn/m})^2m^2}\right) = \Theta\left(\sum_{m=1}^{n} \frac{e^{-cn/m}}{m^2}\right) = \Theta\left(\frac{1}{cn}\right). 
\end{align}
To see the last bound, observe that, for $k \ge 1$,  
$\sum_{m=cn/(k+1)}^{cn/k} \frac{e^{-cn/m}}{m^2} = \Theta(\frac{e^{-k}}{cn})$. Similarly, 
for $1\le k\le 1/c$, $\sum_{m=cnk}^{cn(k+1)} \frac{e^{-cn/m}}{m^2} = \Theta(\frac{e^{-1/k}}{cnk^2})$. The desired bound follows from summing these estimates over $k$. 

By a similar argument, the sum of the absolute centered third moments of the $Y_m/m$ is 
\begin{align}\label{eq:3rd-mm}
    \sum_{m=1}^{n} \left[\frac{1}{1+e^{cn/m}} \left|\frac{1}{m}-\frac{1}{m(1+e^{cn/m})}\right|^3 + \left(1-\frac{1}{1+e^{cn/m}}\right) \left|-\frac{1}{m(1+e^{cn/m})}\right|^3 \right] &= O\left(\sum_{m=1}^{n} \frac{e^{-cn/m}}{m^3}\right) \nonumber\\
    &= O\left(\frac{1}{(cn)^2}\right). 
\end{align}
The Berry--Esseen bound, Lemma \ref{lem:B-E}, then yields that, for $g\sim \mathcal{N}(0,1)$, 
\begin{equation}\label{eq:pZ}
    \pr(E=1) = \pr[Z\le x] = \pr[g\le 0] + O((cn)^{-1/2}) = 1/2 + O((cn)^{-1/2}),
\end{equation}
where we used that $\mathbb{E}[Z]=x$, together with (\ref{eq:2nd-mm}) and (\ref{eq:3rd-mm}).

We next bound $H(Y_1,\dots,Y_n | E=0) \le \sum_{m=1}^{n} H(Y_m|E=0)$. To bound the summands, we note by Bayes' rule that  \[\pr(Y_m = 1|E=0) = \frac{\pr(E=0|Y_m=1)\pr(Y_m=1)}{\pr(E=0)}\]
and we will use a similar argument with the Berry--Esseen bound to show that $\pr(Y_m = 1|E=0)$ is close to $\pr(Y_m=1)$. 
Indeed, the calculations above similarly yield that the random variable $Z'_m=Z-Y_m/m+1/m$ is a sum of independent random variables with $\mathbb{E}Z'_m = x + \frac{1}{m} - \frac{1}{m(1+e^{cn/m})}$, $\mathrm{Var}(Z'_m) = \Theta(1/(cn))$ and the sum of absolute centered third moments $O(1/(cn)^2)$. 
By Lemma \ref{lem:B-E}, for $g\sim \mathcal{N}(0,1)$, 
\begin{align*}
    \pr(E = 0 | Y_m=1) &= \pr(Z'_m > x) \\
    &= \pr\left(g > -\frac{1/m-1/(m(1+e^{cn/m}))}{\mathrm{Var}(Z'_m)^{1/2}}\right) + O((cn)^{-1/2}) \\
    &= 1/2 + O\left((cn)^{-1/2}+\frac{\sqrt{cn}}{m}\right),
\end{align*}
assuming that $m > 10\sqrt{cn}$ for the last bound, where we used the simple estimate $\pr(g > z) = \frac{1}{2} + O(z)$ for $|z| \le 1$. 
Therefore,
\begin{equation*}
    \left|\frac{\pr(E=0|Y_m=1)}{\pr(E=0)} - 1 \right| = \left|\frac{1/2 + O\left((cn)^{-1/2}+\frac{\sqrt{cn}}{m}\right)}{1/2 + O((cn)^{-1/2})} - 1 \right| \le O\left(\frac{\sqrt{cn}}{m} + \frac{1}{\sqrt{cn}}\right). 
\end{equation*}
Thus, by Bayes' rule,
\begin{equation}\label{eq:prob-bound-1}
    \frac{|\pr(Y_m=1|E=0) - \pr(Y_m=1)|}{\pr(Y_m=1)} \le O\left(\frac{\sqrt{cn}}{m} + \frac{1}{\sqrt{cn}}\right). 
\end{equation}

From (\ref{eq:prob-bound-1}), we have
\[
    \pr(Y_m=1) \left(1 - O\left(\frac{\sqrt{cn}}{m} + \frac{1}{\sqrt{cn}}\right)\right)\le \pr(Y_m=1|E=0) \le \pr(Y_m=1) \left(1 + O\left(\frac{\sqrt{cn}}{m} + \frac{1}{\sqrt{cn}}\right)\right).
\]
Since $h'(p) = \log \frac{1-p}{p} \le \log \frac 1p$, we have that 
\begin{align*}
    &h(\pr(Y_m=1|E=0)) \\
    &\le h(\pr(Y_m=1)) + \left (\log \frac{1}{\pr(Y_m=1) \left(1 - O\left(\frac{\sqrt{cn}}{m} + \frac{1}{\sqrt{cn}}\right)\right)}\right ) \left | \pr(Y_m=1|E=0)-\pr(Y_m=1)\right |\\
    &\le h(\pr(Y_m=1)) + \left (\log \frac{1}{\pr(Y_m=1) \left(1 - O\left(\frac{\sqrt{cn}}{m} + \frac{1}{\sqrt{cn}}\right)\right)}\right ) \cdot O\left(\frac{\sqrt{cn}}{m} + \frac{1}{\sqrt{cn}}\right) \pr(Y_m=1) \\
    &\le h(\pr(Y_m=1)) +  O\left(\frac{\sqrt{cn}}{m} + \frac{1}{\sqrt{cn}}\right) \pr(Y_m=1) \log \frac{1}{\pr(Y_m=1)},
\end{align*}
where in the last inequality we used $\pr(Y_m=1) \le 1/2$, so $\log \frac{1}{\pr(Y_m=1) \left(1 - O\left(\frac{\sqrt{cn}}{m} + \frac{1}{\sqrt{cn}}\right)\right)} = O\left(\log \frac{1}{\pr(Y_m=1)}\right)$. 
Therefore, 
\begin{align*}
    H(Y_1,\dots,Y_n|E=0) &\le \sum_{m=1}^{n} H(Y_m | E=0)\\
    &\le 10\sqrt{cn} + \sum_{m>10\sqrt{cn}} h(\pr(Y_m=1|E=0)) \\
    &\le 10\sqrt{cn} + \sum_{m>10\sqrt{cn}} \left(H(Y_m) + O\left(\frac{\sqrt{cn}}{m} + \frac{1}{\sqrt{cn}}\right) \pr(Y_m=1) \log \frac{1}{\pr(Y_m=1)}\right)\\
    &\le H(Y_1,\dots,Y_n) + 10\sqrt{cn} + \sum_{m>10\sqrt{cn}}O\left(\frac{\sqrt{cn}}{m} + \frac{1}{\sqrt{cn}}\right) \pr(Y_m=1) \log \frac{1}{\pr(Y_m=1)}\\
    &\le H(Y_1,\dots,Y_n) + 10\sqrt{cn} + \sum_{m>10\sqrt{cn}} O\left(\frac{\sqrt{cn}}{m} + \frac{1}{\sqrt{cn}}\right) \frac{cn/m}{1+e^{cn/m}}\\
    &\le H(Y_1,\dots,Y_n) + O(\sqrt{n/c}). 
\end{align*}
Again, for the last estimate, we note that, for $k \ge 1$, $\sum_{m=cn/(k+1)}^{cn/k} \left(\frac{\sqrt{cn}}{m} + \frac{1}{\sqrt{cn}}\right) \frac{cn/m}{1+e^{cn/m}} \le O\left(\frac{cn}{k^2} e^{-k}\frac{k^2}{\sqrt{cn}}\right) = O(e^{-k}\sqrt{cn})$  and, for $1\le k\le 1/c$, $\sum_{m=cnk}^{cn(k+1)} \left(\frac{\sqrt{cn}}{m} + \frac{1}{\sqrt{cn}}\right) \frac{cn/m}{1+e^{cn/m}} \le O\left(cn\frac{1}{k\sqrt{cn}}\right) \le O\left(\frac{1}{k}\sqrt{cn}\right)$. Summing over $k$, we thus have 
\[
    \sum_{m>10\sqrt{cn}} O\left(\frac{\sqrt{cn}}{m} + \frac{1}{\sqrt{cn}}\right) \frac{cn/m}{1+e^{cn/m}} = O(\sqrt{n/c}).
\]
Combining with (\ref{eq:ent-decomp}), and noting that $H(E)\le 1$ and $\pr(Z\le x) = 1/2 + O((cn)^{-1/2})$ by (\ref{eq:pZ}), we obtain that 
\[
    H(Y_1,\dots,Y_n|E=1) \ge (1-O((cn)^{-1/2}))(H(P(x)) - O(\sqrt{n/c})) = H(P(x)) - O(\sqrt{n/c}).
\]
Using that for any random variable $X$ we have $H(X)\le \log |\mathrm{supp}(X)|$, we obtain that the number of subsets $A\subseteq [n]$ with $s(A) \le x$ is at least $2^{H(P(x)) - O(\sqrt{n/c})}$. This implies that the number of subsets $A\subseteq [n]\setminus U$ with $s(A)\le x$ is at least 
\[
    2^{-(n-|U|)} 2^{H(P(x)) - O(\sqrt{n/c})},
\]
as required. 
\end{proof}

\section{Subset sums of modular inverses}

The main technical tool we still need is the following result, which says that if $q$ is a large prime power and we take a dense subset $I$ of the interval $[q^{\eps}, 2q^{\eps}]$, then every residue class mod $q$ can be written as the sum of a small number of reciprocals of elements of $I$. 
Roughly speaking, this allows us to cancel out any particular prime power from the denominator of a fraction in the absorption step of the proof of Theorem \ref{thm:main-count}. Given a set $A$ of integers, we will use the notation $\Sigma^{[s]}(A)$ for the collection of sums of subsets of $A$ of size at most $s$. 

\begin{theorem}
\label{thm:main}
Let $\delta, \eps>0$ and let $q$ be a prime power which is sufficiently large in terms of $\delta, \eps$. If $I$ is a subset of $[q^{\eps}, 2q^{\eps}]$ consisting of elements coprime to $q$ with $|I| \ge \delta q^{\eps}$, then $\Sigma^{[s]}(I^{-1}) \pmod q = \bZ_q$ for $s = q^{\eps/2}$.  
\end{theorem}

In the proof of Theorem \ref{thm:main}, we will make use of the following key result from \cite{CFP}. Recall that a \textit{generalized arithmetic progression (henceforth GAP) $P$ of dimension $k$} is a set of integers $\{x_0 + \ell_1 x_1 + \ell_2 x_2 + \cdots + \ell_k x_k | 0 \le \ell_1 < L_1, \ldots, 0 \le \ell_k < L_k\}$. A GAP is called \textit{proper} if it has size exactly $L_1 L_2\cdots L_k$. We say that $P$ is \textit{homogeneous} if $\gcd(x_1,\dots,x_k)$ divides $x_0$. For a natural number $t$, we define $tP$ to be the $t$-fold sumset of $P$, while if $t$ is a positive real number which is not an integer and $P = \{\sum_{i=1}^k n_i x_i : a_i \le n_i \leq b_i\}$ is a homogeneous GAP, we can generalize the definition by setting $tP = \{\sum_{i=1}^k n_i x_i : ta_i \le n_i \leq tb_i\}$. 

\begin{theorem} \label{thm:hom-AP-build} 
For any $\beta>1$ and $0<\eta<1$, there are positive constants $c$ and $k$ such that the following holds. Let $A$ be a subset of $[n]$ of size $m$ with $n\le m^{\beta}$ and let $s\in[m^{\eta},cm/\log m]$. Then there exists a subset $\hat{A}$ of $A$ of size at least $m-c^{-1}s\log m$ and a proper GAP $P$ of dimension at most $k$ such
that $\hat{A} \cup \{0\}$ is a subset of $P$. Furthermore, there exists
$A'\subseteq \hat{A}$ of size at most $s$ such that $\Sigma(A')$
contains a homogeneous translate of $csP$, where $csP$ is proper.
\end{theorem}

We will also need the following simple variant of Dirichlet's simultaneous approximation theorem. For a residue class $i\pmod q$, we use the notation $\ol{i}$ for the unique integer in $(-q/2,q/2]$ congruent to $i$ modulo $q$. 

\begin{lemma}\label{lem:nbh} 
    Given a prime power $q$, integers $d_1,\dots,d_k$ coprime to $q$ and positive integers $a_1,\dots,a_k$ such that $\prod_{i=1}^{k} a_i = A$, there exists a positive integer $T < q$ and integers $d_1',\ldots, d_k'$ such that $Td_i = d_i' \pmod q$ and $|d_i'| \le 2 (q/a_i) \cdot (A/q)^{1/k}$ for all $i\in [k]$.
\end{lemma}
\begin{proof}
    Let $b_i = 2 (q/a_i) \cdot (A/q)^{1/k}$. Note that $\lfloor \ol {sd_i}/b_i\rfloor$ takes at most $q/b_i$ values as $s$ ranges over $\bZ_q$. By the pigeonhole principle, there exist distinct $s\ne s'$ in $\bZ_q$ such that $\lfloor \ol{sd_i} / b_i\rfloor = \lfloor \ol{s'd_i} / b_i\rfloor$ for all $i\in [k]$, since $\prod_{i=1}^{k} \frac{q}{b_i} = (q/A) 2^{-k}\prod_{i=1}^{k} a_i < q$. Letting $T=s'-s$, we then have that $|\ol{Td_i}| \le b_i = 2 (q/a_i) \cdot (A/q)^{1/k}$.  
\end{proof}

We now proceed to the proof of Theorem \ref{thm:main}. The basic idea is to use Theorem~\ref{thm:hom-AP-build} to argue that there is a large subset $J$ of the set of inverses $I^{-1}$ which is contained in a proper GAP $P$ of bounded dimension $k$ such that $\Sigma^{[s]}(J)$ contains a proper translate of $csP$. 
We then exploit the nature of the set of inverses to argue that $k$ must in fact be $1$, that is, $P$ is simply a progression, from which the required result quickly follows.

\begin{proof}[Proof of Theorem \ref{thm:main}]
    Let $s=q^{\eps/2}$. Let $\ol{I}$ and $\ol{I^{-1}}$ denote the set of integer representations (in $(-q/2,q/2]$) of $I$ and $I^{-1}$. By Theorem \ref{thm:hom-AP-build}, there is $c$ depending only on $\eps$ such that we can find $J \subseteq \ol{I^{-1}}$ of size at least $|I| - c^{-1} s \log |I| = (1-o(1))|I|$ and a proper GAP $P$ of dimension $k=O_{\eps}(1)$ such that $J \cup \{0\}\subseteq P$ and $\Sigma^{[s]}(J)$ contains a translate of $csP$ which is proper.

    By expanding $P$ by a factor of up to $2^k$ if necessary, we can write $P = \sum_{u=1}^{k} [-a_u,a_u]d_u$. Let $A=\prod_{u=1}^{k} a_u$. With these $a_u$ and $d_u$, we apply Lemma \ref{lem:nbh} to find a value of $T$ satisfying the conclusions of that lemma and let $T\cdot P = \{tx: x\in P\}$. Note that, for any $j\in T\cdot P$, $|\ol{j}| \le \sum_{u=1}^{k} a_u |d_u'| \le 2kq(A/q)^{1/k}$. 

    \begin{claim}
        Let $N$ denote the number of solutions to the equation $\ol{i}\cdot \ol{j} = T \pmod q$ with $i\in \ol{I}$ and $j\in T\cdot J$. Then 
        \begin{equation}\label{eq:sol-count}
            |I|/2\le N < q^{C/\log \log q} \cdot 8kq^{\eps}(A/q)^{1/k}. 
        \end{equation}
    \end{claim}
    We first complete the proof of Theorem \ref{thm:main} assuming the claim. 
    From (\ref{eq:sol-count}) and the assumption that $|I| \geq \delta q^{\eps}$, we deduce that 
    \begin{equation*} 
        A \ge \delta^k (16k)^{-k}q^{1-Ck/\log \log q}.  
    \end{equation*}
    On the other hand, since $\Sigma^{[s]}(J)$ contains a translate of $csP$ with $csP$ proper and $\Sigma^{[s]}(J) \subseteq (-sq/2, sq/2]$, we have that
    \begin{equation*} 
        c^ks^k A \le c^ks^k |P| = |csP| \le |\Sigma^{[s]}(J)| \le sq.
    \end{equation*}
    Hence, $A \le c^{-k}qs^{1-k} = c^{-k}q^{1-(k-1)\eps/2}$. Provided $q$ is sufficiently large in terms of $\delta, \eps$, these two estimates on $A$ together imply that $k=1$. Therefore, $csP$ is an arithmetic progression of length $\Omega(As) > q$. Furthermore, $P$ must have common difference coprime with $q$ as $J \subseteq \ol{I^{-1}}$ is contained in $P$. Hence, any translate of $csP$ covers all residue classes in $\bZ_q$. This finishes the proof of the theorem assuming the claim. 

    It remains to verify the claim. Since each $j_0\in J$ has $j_0^{-1}\in I \pmod q$, the number of solutions to $\ol{i} \cdot \ol{j} = T \pmod q$ with $i\in \ol{I}$, $j = T \cdot j_0 \in T \cdot J \subseteq T\cdot P$ is at least $|J| = (1-o(1))|I| \ge |I|/2$. 
Since $I \subseteq [q^{\eps}, 2q^{\eps}]$, we have that 
\[|\ol{i} \cdot \ol{j}| \le 2q^{\eps} \cdot 2kq(A/q)^{1/k}.\]
    As such, if $\ol{i} \cdot \ol{j} = T\pmod q$, then $\ol{i} \cdot \ol{j} = qx+T$, where $0\le |x| \le 4kq^{\eps} (A/q)^{1/k}$. But the number of solutions to the equation $\ol{i} \cdot \ol{j} = qx+T$ with $0\le |x| \le 4kq^{\eps} (A/q)^{1/k}$ is bounded above by
    \begin{equation*}
        \sum_{|x|\le 4kq^{\eps}(A/q)^{1/k}} \tau(qx+T) < q^{C/\log \log q} \cdot 8kq^{\eps}(A/q)^{1/k}, 
    \end{equation*} 
    where $\tau(n)$ denotes the number of divisors of $n$ and we have used the standard bound $\tau(n) \le q^{C/\log \log q}$ for an absolute constant $C$ and all $n\le q^2$. This completes the proof of the claim. 
\end{proof}

\section{Absorption}

We are now ready to prove Theorem \ref{thm:main-count} in the following explicit form. We recall that a positive integer $n$ is {\it $t$-smooth} if all of its prime factors are at most $t$ and {\it $t$-powersmooth} if all of its prime power factors are at most $t$. 

\begin{theorem}\label{thm:main-count-explicit}
    Let $\eps > 0$ be sufficiently small. Then there exists $\xi > 0$ such that if $\eps \le x \le \xi \ln n$ is a rational whose denominator is $(n^{1-\eps}/2)$-powersmooth, then the number of subsets $A\subseteq [n]$ with $x = \sum_{a\in A}1/a$ is at least $2^{c_x n - c_\eps n}$, where $c_\eps \to 0$ as $\eps \to 0$ and
    \[
        c_x := \int_{0}^{1}h\left(\frac{1}{1+e^{\lambda/y}}\right)dy
    \] 
    with $\lambda$ the unique real number such that 
    \[
        \int_{0}^{1}\frac{1}{y(1+e^{\lambda/y})}dy = x.
    \]
\end{theorem}

We first record a simple lemma guaranteeing that most integers at most $n$ are $(n^{1-\eps}/2)$-powersmooth. 

\begin{lemma}\label{lem:powersmooth}
    For $\delta$ sufficiently small, $n$ sufficiently large in terms of $\delta$ and $t = n^{1-\delta}$, at least $(1-2\delta)n$ positive integers at most $n$ are $t$-powersmooth. 
\end{lemma}
\begin{proof}
    It is well known that if $t = n^{u}$, the number of $t$-smooth numbers up to $n$ is asymptotic to $\psi(u)n$, where $\psi$ is the Dickman function taking values in $(0,1)$ for $u\in(0,1)$. Moreover, for $u>1/2$, $\psi(u) = 1+\ln u$. Thus, for $u=1-\delta$ and $t=n^{1-\delta}$, at least $(1+\ln(1-\delta)-o(1))n \ge (1-\frac 32 \delta)n$ positive integers at most $n$ are $n^{1-\delta}$-smooth.

    Among the $t$-smooth numbers, the only ones that are not $t$-powersmooth are those divisible by a prime power $p^{\alpha}$ where $p\le t$ but $p^\alpha > t$. Using the prime number theorem, we may upper bound the total count of such exceptional smooth numbers by
    \[
    \sum_{p \le t} \left\lfloor \frac{n}{t} \right\rfloor = \pi(t) \left\lfloor \frac{n}{t} \right\rfloor = o(n).
    \]
    Hence, at least $(1-\frac 32 \delta)n - o(n) \geq (1-2\delta)n$ positive integers at most $n$ are $t$-powersmooth.
\end{proof}

Finally, we prove Theorem \ref{thm:main-count-explicit}. Recall the notation that, for $A\subseteq [n]$, $s(A) = \sum_{a\in A}1/a$.

\begin{proof}[Proof of Theorem \ref{thm:main-count-explicit}.]
By choosing $\xi$ suitably, we can assume that $n$ is sufficiently large in terms of $\eps$. Let $L$ be sufficiently large, assuming in particular that Theorem~\ref{thm:main} applies for $\eps$ as in the statement of the theorem, $\delta = \frac 12$ and all $q > L$. 
Let $K$ denote the least common multiple of all prime powers at most $L$. 
We first reserve the set $\cR$ of multiples of $K$ in $[n]$. 
Let $\cP(q) = q\cdot [q^{\eps}, 2q^{\eps}] \setminus \cR$ and $\cP = \bigcup_{q} \cP(q)$, where $q$ ranges over all prime powers at most $n^{1-\eps}/2$. Here $q\cdot S = \{qs:s\in S\}$ and the notation $[q^{\eps},2q^{\eps}]$ refers to the set of integers in this interval. 
Let $\cS$ denote the set of $(n^{1-\eps}/2)$-powersmooth numbers at most $n$ and $\cU = \cS \setminus (\cR \cup \cP)$. Lemma \ref{lem:powersmooth} implies that $|\cS| \ge (1-O(\eps))n$ and we also have that $|\cP \cap [n]| \le \sum_{q\le n^{1-\eps}}2q^{\eps} \le 4 \frac{n}{\ln n}$. Thus, 
\begin{equation} \label{eq:cU-est}
    n-|\cU| \le n/K + O(\eps n).
\end{equation}

Let $\eta>0$ be a constant to be chosen later. By Lemma \ref{lem:main-ent}, applied with $x$ replaced by $(1-\eta)x$, we can find many subsets of $\cU$ whose sums of reciprocals are at most $(1-\eta)x$. Indeed, the number of such subsets is at least 
\[
    2^{H(P((1-\eta)x)) - (n - |\cU|) - O(\sqrt{n/c_{(1-\eta)x,n}})}. 
\]
Fix one such sum corresponding to a set $A_0 \subseteq [n]$ and let $x_0 = x - s(A_0) \ge \eta x$. Consider the following procedure, where at each step $i$ we have a real number $x_i$ and a set $A_i$ for which $x_i = x - s(A_i)$: 
\begin{enumerate}
    \item In decreasing order over the prime powers larger than $L$, consider the largest prime power $q = q_i \le n^{1-\eps}/2$ of a prime $p = p_i$  which appears as a factor of the denominator of $x_{i}$. 
    We then find $B_i \subseteq (1/q)\cdot \cP(q)$ of size at most $q^{\eps/2}$ such that, for $x_{i} = \frac{u_{i}}{v_{i}}$ with $u_{i}, v_{i}$ coprime, $s(B_i) = -\frac{u_{i}}{v_i/q} \pmod q$. 
    We say that step $i$ succeeds if we can find such a $B_i$. If it does succeed, we update $x_{i+1} = x_i - s(q \cdot B_i)$ and $A_{i+1} = A_i \cup q \cdot B_i$, noting by our choice that no nonzero power of $p$ divides the denominator of $x_{i+1}$. Furthermore, any new prime power divisor of the denominator of $x_{i+1}$ is at most $2q^{\eps}$. 
    
    \item We iterate until all the prime powers $q_i>L$ have been processed. At this point, the final output $x_{f}$ is a rational number whose denominator is $L$-powersmooth. We then find a subset of the reservoir $\cR$ whose sum of inverses is equal to $x_{f}$.
\end{enumerate}

The following claim guarantees that the procedure above succeeds. 
\begin{claim}
    For each $i$, step $i$ succeeds. Furthermore, $s(q_i\cdot B_i) \le q_i^{-1-\eps/2}$ and,  
    for some absolute constant $C>0$,  
    \begin{equation}\label{eq:xfx0}
        |x_{f}-x_0| \le C\eps^{-1}L^{-\eps/2}. 
    \end{equation}
\end{claim}
\begin{proof}
    Theorem \ref{thm:main} implies immediately that step $i$ always succeeds. Furthermore, by our choice of $B_i$,
    \[
        s(q_i \cdot B_i) \le \frac{q_i^{\eps/2}}{q_i^{1+\eps}} = q_i^{-1-\eps/2}.   
    \]
    The estimate (\ref{eq:xfx0}) follows since the $q_i$ are distinct integers between $L$ and $n$, so 
    \[
        |x_f - x_0| = \sum_i s(q_i \cdot B_i) \le \sum_{L<q_i<n} q_i^{-1-\eps/2} \le \sum_{L<m<n} m^{-1-\eps/2} < O(\eps^{-1}L^{-\eps/2}), 
    \]
    as required.
    \end{proof}

We ensure that $L,\eta>0$ are chosen (depending on $\eps$) so that $C\eps^{-1}L^{-\eps/2} < \eta x/2$. From (\ref{eq:xfx0}), we have that $x_f$ is a positive rational number at most $x$ whose denominator is $L$-powersmooth. As such, we have that $Kx_f$ is a positive integer with $Kx_f \le Kx$. We now note that there exists a subset $D \subseteq [n/K]$ such that $s(D)= K x_f$, where we use the assumption that $x\le \xi \ln n$ for $\xi$ chosen sufficiently small in $\eps$, so that $Kx < \eps \ln(n/K)$. 
To see that this is the case, one may, for example, make use of Croot's result~\cite{C01} that 1 can always be written as the sum of reciprocals of numbers from any interval of the form $[t, (e+o(1))t]$.
This allows us to iteratively remove $K x_f$ disjoint subsets from $[n/K]$, the sum of the reciprocals of each of which is 1.
We then set their union to be $D$, noting that $K\cdot D \subseteq \cR$ is disjoint from $\cU$ and $\cP$. 

We then have $x = \sum_{d\in D}\frac{1}{Kd} + \sum_i s(q_i \cdot B_i) + s(A_0)$ and the number of such distinct representations is at least the number of choices for $A_0$, which is bounded below by 
\[
    2^{H(P((1-\eta)x)) - (n - |\cU|) - O(\sqrt{n/c_{(1-\eta)x,n}})} \ge 2^{H(P(x)) - c_\eps n},
\]
for an appropriate constant $c_\eps$ with $c_\eps \to 0$ as $\eps \to 0$, where we have used Lemma \ref{lem:main-ent} and (\ref{eq:cU-est}). This completes the proof of Theorem \ref{thm:main-count-explicit}. 
\end{proof}

\section*{Note added}
As we completed this paper, we learned that a result similar to our Theorem~\ref{thm:main-count} 
was obtained simultaneously and independently, though using rather different methods, by Yang P. Liu and Mehtaab Sawhney~\cite{LM24}.

\section*{Acknowledgments} 
We are grateful to the American Institute of Mathematics for hosting the SQuaREs project at which this work was initiated. We are also indebted to Zachary Chase and the user Lucia for bringing the MathOverflow post~\cite{T17} to our attention.

\bibliographystyle{amsplain}


\begin{dajauthors}
\begin{authorinfo}[dc]
  David Conlon\\
  Department of Mathematics\\
  California Institute of Technology\\
  Pasadena, USA\\
  dconlon@caltech.edu \\
  \url{http://www.its.caltech.edu/~dconlon/}
\end{authorinfo}
\begin{authorinfo}[jf]
  Jacob Fox\\
  Department of Mathematics\\ 
  Stanford University\\
  Stanford, USA\\
  jacobfox@stanford.edu \\
  \url{https://stanford.edu/~jacobfox/}
\end{authorinfo}
\begin{authorinfo}[xh]
  Xiaoyu He\\
  School of Mathematics\\ 
  Georgia Institute of Technology\\ 
  Atlanta, USA\\
  xhe399@gatech.edu\\
  \url{https://alkjash.github.io}
\end{authorinfo}
\begin{authorinfo}[dm]
  Dhruv Mubayi\\
  Department of Mathematics, Statistics and Computer Science\\ 
  University of Illinois Chicago\\ 
  Chicago, USA\\
  mubayi@uic.edu \\
  \url{https://homepages.math.uic.edu/~mubayi/}
\end{authorinfo}
\begin{authorinfo}[htp]
  Huy Tuan Pham\\
  Department of Mathematics\\
  California Institute of Technology\\
  Pasadena, USA\\
  htpham@caltech.edu \\
  \url{https://huytuanpham.github.io}
\end{authorinfo}
\begin{authorinfo}[as]
  Andrew Suk\\
  Department of Mathematics\\ 
  University of California at San Diego\\ 
  La Jolla, USA\\
  asuk@ucsd.edu \\
  \url{https://mathweb.ucsd.edu/~asuk/}
\end{authorinfo}
\begin{authorinfo}[jv]
  Jacques Verstra\"ete\\
  Department of Mathematics\\ 
  University of California at San Diego\\ 
  La Jolla, USA\\
  jacques@ucsd.edu \\
  \url{https://annatar0.wixsite.com/website-4}
\end{authorinfo}
\end{dajauthors}

\end{document}